\documentclass{amsart}

\usepackage{amsmath, amsthm, amssymb, mathtools}
\usepackage[margin=1.5in]{geometry}
\usepackage{comment}
\usepackage{newverbs}
\usepackage{parskip}
\usepackage{changepage}
\usepackage{color}
\usepackage[usenames,dvipsnames,svgnames,table]{xcolor}
\usepackage[normalem]{ulem}

\newtheorem*{thm*}{Theorem}
\newtheorem{thm}{Theorem}[section]
\newtheorem{cor}[thm]{Corollary}
\newtheorem{lemma}[thm]{Lemma}
\newtheorem{prop}[thm]{Proposition}
\newtheorem*{prop*}{Proposition}

\theoremstyle{definition}
\newtheorem{defn}[thm]{Definition}

\theoremstyle{remark}
\newtheorem{remark}[thm]{Remark}
\newtheorem*{remark*}{Remark}

\def\NR{\mathit{NR}}
\DeclareMathOperator{\Ap}{Ap}

\newcommand{\NZ}{\mathbb{N}_0}
\newcommand{\atext}[1]{\text{\quad #1\quad}}

\newcommand{\gens}[1]{\langle #1 \rangle}

\title{Characterizations of numerical semigroup complements via Ap\'ery sets}
\author{T. Alden Gassert}
\address{Department of Mathematics, Western New England University, Springfield, MA 01119 \linebreak  and \linebreak   Department of Mathematics and Computer Science, Hobart and William Smith Colleges, Geneva, NY 14456}
\email{gassert@hws.edu}
\author{Caleb McKinley Shor}
\address{Department of Mathematics, Western New England University, Springfield, MA 01119}
\email{cshor@wne.edu}

\begin{document}

\begin{abstract}
In this paper, we generalize the work of Tuenter to give an identity which completely characterizes the complement of a numerical semigroup in terms of its Ap\'ery sets. Using this result, we compute the $m$th power Sylvester and alternating Sylvester sums for free numerical semigroups. Explicit formulas are given for small $m$.
\end{abstract}

\maketitle

\section{Introduction, preliminaries}

Let $\mathbb{N}_0$ denote the set of non-negative integers.  A \textit{numerical semigroup} $S$ is a subset of $\mathbb{N}_0$ that is closed under addition, contains 0, and has finite complement in $\mathbb{N}_0$. For any nonempty $G\subset\mathbb{N}_0$, let $\gens{G} = \{\sum_{i = 1}^k n_ig_i : k \in \mathbb{N}_0, n_i\in\mathbb{N}_0, g_i \in G\}$, the set of all non-negative linear combinations of elements of $G$.  It is known that $\gens{G}$ is a numerical semigroup if and only if $\gcd(G)=1$.  Furthermore, given any numerical semigroup $S$, there is some finite generating set $G$ such that $S=\langle G\rangle$. These results are well known. For a comprehensive introduction to numerical semigroups, see \cite[Chapter 1]{RosalesGarciaSanchez09}. 

For $S$ a numerical semigroup with $S=\gens{G}$, the complement of $S$ in $\mathbb{N}_0$ is the set of positive integers that are \emph{not representable} in terms of $G$, and we denote this finite set $\NR=\mathbb{N}_0\setminus S$. The \emph{genus} of $S$ is $g(S)=\#\NR$, and the \emph{Frobenius number} of $S$ is $F(S)=\max(\mathbb{Z}\setminus S)$, which is $\max(\NR)$ when $\NR\neq\emptyset$.

The problem of computing the genus and Frobenius number of a numerical semigroup dates back to the late 19th century. In \cite{Sylvester1882}, Sylvester solved the problem when $S$ is generated by two integers. In the 20th century, the idea of the Ap\'ery set of an element in a numerical semigroup was introduced in \cite{Apery1946}. If one knows the Ap\'ery set of an element, then one can then compute the Frobenius number and genus of the numerical semigroup. The formula for the Frobenius number is attributed to Brauer and Shockley in \cite{Brauer1962}, and the genus result is attributed to Selmer in \cite{Selmer1977}. In general, the Ap\'ery set of an element in a numerical semigroup is difficult to compute.

In \cite{Tuenter06}, Tuenter remarks that one can recover a set if one knows the images of the set under sufficiently many mappings.  In particular, he notes that the Sylvester sums of $\NR$ are sufficient to completely determine $\NR$, and he gives an identity (\cite[Theorem 2.1]{Tuenter06}) that allows for a simple and efficient computation for these sums in the case that $\NR = \NZ \setminus \gens{a,b}$ (that is, $\NR$ is the complement of a semigroup generated by two elements).  In this paper, we give a reformulation of Tuenter's identity for general semigroups in terms of Ap\'ery sets (Theorem \ref{thm:tuenter-apery}).  Ap\'ery sets (at least for certain values in $S$) are well understood when $S$ is free, thus our identity gives a complete characterization of $\NR$ for free numerical semigroups.

This paper is organized as follows. In Section \ref{sec:identity}, we define Ap\'ery sets and use them to prove Theorem \ref{thm:tuenter-apery}, the identity which characterizes $\NR$. We then use the identity to produce a well-known formula for the Hilbert series of a numerical semigroup in terms of the Ap\'ery set of any element in the semigroup. In Section \ref{sec:smooth}, we describe free numerical semigroups and their Ap\'ery sets. Such numerical semigroups include those generated by compound sequences, geometric sequences, supersymmetric sequences, and two-element sets. Finally, in Section \ref{sec:sylvester sum}, we give an explicit version of the identity for free numerical semigroups, and we use it to compute the $m$th power and alternating $m$th power Sylvester sums of these semigroups.

\section{An identity for Ap\'ery sets}\label{sec:identity}

\begin{defn}
Let $S$ be a numerical semigroup. For any nonzero $t\in S$, the \emph{Ap\'ery set} of $t$ in $S$ is defined as 
\[\Ap(S;t)=\left\{s\in S : s-t\not\in S \right\}.\]  
Equivalently, if we let $w_i=\min\{s\in S : s\equiv i\pmod{t}\}$, then $\Ap(S;t)=\left\{w_i : 0\le i<t\right\}$. In other words,  $\Ap(S;t)$ consists of the minimal element of $S$ in each congruence class modulo $t$.
\end{defn}

Ap\'ery sets are incredibly helpful in understanding properties of a numerical semigroup and its complement.  For instance, given a numerical semigroup $S$ and some $t\in S$, one has the following formulas for the genus \cite[page 3]{Selmer1977} and  Frobenius number \cite[Lemma 3]{Brauer1962} of $S$ in terms of $\Ap(S;t)$:
\begin{equation}\label{eqn:selmer-genus}
g(S)=-\frac{t-1}{2}+\frac{1}{t}\sum_{n\in\Ap(S;t)}n
\end{equation} 
and 
\begin{equation}\label{eqn:brauer-frobenius} 
F(S)=\max(\Ap(S;t))-t.
\end{equation}

With Ap\'ery sets, we have an identity which can be seen as a generalization of \cite[Theorem 3.3]{GassertShor16} (Corollary \ref{cor:compound-tuenter} in this paper), which itself is a generalization of \cite[Theorem 2.1]{Tuenter06} (Corollary \ref{cor:tuenter} in this paper). We first need a lemma.
\begin{lemma}\label{lemma:set-equalities}
Let $S$ be a numerical semigroup with complement $\NR=\mathbb{N}_0\setminus S$. For any nonzero $t\in S$, let $(\NR+t):=\{n+t:n\in\NR\}$. For any $a,b\in\mathbb{Z}$, let $I_{a,b}:=\{n\in\mathbb{Z} : a\le n<b\}$. Finally, let $\Ap_t:=\{s\in\Ap(S;t):s<t\}=\Ap(S;t)\cap I_{0,t}.$
Then
\begin{equation}\label{eqn:set-equality-1}(\NR+t)\setminus\NR=\Ap(S;t)\setminus\Ap_t\end{equation}
and 
\begin{equation}\label{eqn:set-equality-2} \NR\setminus(\NR+t)=I_{0,t}\setminus\Ap_t.\end{equation}
\end{lemma}
\begin{proof}
To prove these equalities, we will show the sets are subsets of each other.

For Equation \eqref{eqn:set-equality-1}, to show inclusion we suppose $s\in (\NR+t)\setminus\NR$. Note that $t>0$. Since $s\in(\NR+t)$, $s=n+t$ for some gap $n\in\NR$. The gaps of $S$ are positive integers, so $s>t$. Since $s\not\in\NR$ and $s>0$, $s\in S$. Then, with $s\in S$ and $s-t\in \NR$, we have $s\in\Ap(S;t)$. Finally, since $s>t$, $s\not\in\Ap_t$. Thus $s\in\Ap(S;t)\setminus\Ap_t$, so $(\NR+t)\setminus\NR\subseteq\Ap(S;t)\setminus\Ap_t$.

For the reverse inclusion, we suppose $s\in\Ap(S;t)\setminus\Ap_t$. Since $s\in\Ap(S;t)$, we have $s\in S$ and $s-t\not\in S$. Since $s\in\Ap(S;t)$ and $s\not\in\Ap_t$, $s>t$, so $s-t>0$. Thus $s-t\in\NR$, so $s\in(\NR+t)$. Finally, since $s\in S$, $s\not\in\NR$. Therefore $s\in(\NR+t)\setminus\NR$, so $\Ap(S;t)\setminus\Ap_t\subseteq(\NR+t)\setminus\NR$. Combined with the previous paragraph, we have $(\NR+t)\setminus\NR = \Ap(S;t)\setminus\Ap_t$.

For Equation \eqref{eqn:set-equality-2}, to show inclusion we suppose $n\in\NR\setminus(\NR+t)$. Let $m=n-t$. Since $n\in\NR$ and $t\in S$, we must have $m\not\in S$. Since $n\not\in(\NR+t)$, $m\not\in\NR$, so $m<0$. Thus $1\le n<t$. Furthermore, $n\in\NR$ so $n\not\in\Ap_t$. Thus $n\in I _{0,t}\setminus\Ap_t$, so $\NR\setminus(\NR+t)\subseteq I _{0,t}\setminus\Ap_t$.

For the reverse inclusion, we suppose $n\in I_{0,t}\setminus\Ap_t$. Since $0\le n<t$, if $n\in S$ then $n$ is the minimal element of $S$ in its congruence class modulo $t$, so $n\in\Ap(S;t)$ and therefore $n\in\Ap_t$. However, since $n\not\in\Ap_t$, this is impossible so we must have $n\not\in S$, so $n\in\NR$. And since $n-t<0$, we have $n\not\in(\NR+t)$. Therefore $n\in\NR\setminus(\NR+t)$, so $I_{0,t}\setminus\Ap_t\subseteq\NR\setminus(\NR+t)$. Combined with the previous paragraph, we have $\NR\setminus(\NR+t)=I_{0,t}\setminus\Ap_t$.
\end{proof}
We now have our main result.
\begin{thm}\label{thm:tuenter-apery}
Let $S$ be a numerical semigroup with complement $\NR=\mathbb{N}_0\setminus S$.  For any nonzero $t\in S$ and any function $f$ defined on $\mathbb{N}_0$, 
\begin{equation}\label{eqn:tuenter-apery}
\sum_{n\in \NR} [f(n+t)-f(n)] = \sum_{n\in \Ap(S;t)}f(n)-\sum_{n=0}^{t-1}f(n).
\end{equation}
\end{thm}

\begin{proof}
By Lemma \ref{lemma:set-equalities}, $(\NR+t)\setminus\NR=\Ap(S;t)\setminus\Ap_t,$ so for any function $f$ defined on these sets, we have 
\begin{equation}\label{eqn:summation-eqn-1}
\sum\limits_{n\in(\NR+t)\setminus\NR}f(n)=\sum\limits_{n\in\Ap(S;t)\setminus\Ap_t}f(n).
\end{equation}
Since $\Ap_t$ and $\Ap(S;t)$ are finite sets and $\Ap_t\subseteq\Ap(S;t)$, we can rewrite Equation~\eqref{eqn:summation-eqn-1} as
\begin{equation}\label{eqn:summation-eqn-1-cleaned-up}
\sum\limits_{n\in(\NR+t)\setminus\NR}f(n)=\sum\limits_{n\in\Ap(S;t)}f(n)-\sum\limits_{n\in\Ap_t}f(n).
\end{equation}
Similarly, by Lemma \ref{lemma:set-equalities}, $\NR\setminus(\NR+t)=I_{0,t}\setminus\Ap_t,$ so for any function $f$ defined on these sets, we have
\begin{equation}\label{eqn:summation-eqn-2}
\sum\limits_{n\in\NR\setminus(\NR+t)}f(n)=\sum\limits_{n\in I_{0,t}\setminus\Ap_t}f(n).
\end{equation}
Since $\Ap_t$ and $I_{0,t}$ are finite and $\Ap_t\subseteq I_{0,t}$, we can rewrite Equation~\eqref{eqn:summation-eqn-2} as
\begin{equation}\label{eqn:summation-eqn-2-cleaned-up}
\sum\limits_{n\in\NR\setminus(\NR+t)}f(n)=\sum\limits_{n\in I_{0,t}}f(n)-\sum\limits_{n\in \Ap_t}f(n).
\end{equation}

Then, since 
\begin{equation}\label{eqn:next-eqn}
\sum\limits_{n\in(\NR+t)}f(n)-\sum\limits_{n\in\NR}f(n) = \sum\limits_{n\in(\NR+t)\setminus\NR}f(n)-\sum\limits_{n\in\NR\setminus(\NR+t)}f(n),
\end{equation}
we simplify the right side of Equation \eqref{eqn:next-eqn} using Equations \eqref{eqn:summation-eqn-1-cleaned-up} and \eqref{eqn:summation-eqn-2-cleaned-up} to conclude
\begin{equation}
\sum\limits_{n\in(\NR+t)}f(n)-\sum\limits_{n\in\NR}f(n)=\sum\limits_{n\in\Ap(S;t)}f(n)-\sum\limits_{n\in I_{0,t}}f(n).
\end{equation}
Note that any function $f$ defined on $\mathbb{N}_0$ will necessarily be defined on the sets described above. The desired result follows immediately.\end{proof}

\begin{remark}
Since $0\in \Ap(S;t)$, we can exclude 0 from both summations on the right to write Equation~\eqref{eqn:tuenter-apery} as 
\begin{equation}\label{eqn:tuenter-apery-no-0}
\sum_{n\in\NR}[f(n+t)-f(n)]=\sum_{\substack{n\in \Ap(S;t) \\ 
n\neq 0}}f(n)-\sum_{n=1}^{t-1}f(n)
\end{equation}
and thereby apply this identity to functions $f$ defined on $\mathbb{N}$.
\end{remark}

\begin{remark}
Since $\Ap(S;t)$ contains one element from each congruence class modulo $t$, we can write Equation~\eqref{eqn:tuenter-apery} as 
\begin{equation}\label{eqn:tuenter-apery-congruences}
\sum_{n\in\NR}[f(n+t)-f(n)] = \sum_{n\in \Ap(S;t)}[f(n)-f(\overline{n})],
\end{equation} 
where $\overline{n}$ denotes the reduction of $n$ modulo $t$. In other words, given $n\in\mathbb{Z}$, the number $\overline{n}$ is the least non-negative integer congruent to $n$ modulo $t$. 
As with Equation~\eqref{eqn:tuenter-apery-no-0}, we can remove $n=0$ from the summation on the right if desired.
\end{remark}

If $S=\gens{G}$ for $G=\{a,b\}$ with $\gcd(a,b)=1$, one finds that $\Ap(S;a)=\{nb : 0\le n<a\}$. (See \cite{Selmer1977}.) In that case, we obtain Tuenter's result.

\begin{cor}[{\cite[Theorem 2.1]{Tuenter06}}]\label{cor:tuenter}
For $G=\{a,b\}$ with $\gcd(a,b)=1$, let $\NR=\mathbb{N}_0\setminus\gens{G}$. Then for any function $f$ defined on $\mathbb{N}$,
\begin{equation}\label{eqn:tuenter-original}
\sum_{n\in\NR}[f(n+a)-f(n)] = \sum_{n=1}^{a-1}[f(nb)-f(n)].
\end{equation}
\end{cor}

Immediately after \cite[Theorem 2.1]{Tuenter06}, Tuenter mentions that Equation~\eqref{eqn:tuenter-original} completely characterizes the set $\NR$ when $G=\{a,b\}$, which is to say that all properties of $\NR$ can be derived from this equation.  The same argument and result holds here.  Equation~\eqref{eqn:tuenter-apery} completely characterizes the set $\NR$ for any set $G$.

\subsection{Applications to the genus and Hilbert series}
We conclude this section with two applications of Theorem~\ref{thm:tuenter-apery}.

\subsubsection{A formula for the genus}
We first use $f(n)=n$ in Equation~\eqref{eqn:tuenter-apery} to re-derive Selmer's genus result (Equation~\eqref{eqn:selmer-genus}). We find \[\sum_{n\in\NR}t = \sum_{n\in\Ap(S;t)}n-\sum_{m=0}^{t-1}m,\] so \[g(S)=\#\NR = -\frac{t-1}{2}+\frac{1}{t}\sum_{n\in \Ap(S;t)}n.\] 

\subsubsection{A formula for the Hilbert series}
The Hilbert series $H_S(x)$ of a numerical semigroup $S$ is defined to be the formal power series \[H_S(x) = \sum_{n\in S}x^n=\frac{1}{1-x}-\sum_{n\in\NR}x^n.\]  Using $f(n)=x^n$ with Equation \eqref{eqn:tuenter-apery}, we obtain \[(x^t-1)\sum_{n\in\NR}x^n = \sum_{n\in\Ap(S;t)}x^n-\frac{1-x^t}{1-x},\] 
so \[H_S(x)=\frac{1}{1-x^t}\sum_{n\in \Ap(S;t)}x^n.\] For more on Hilbert series, including this identity, see \cite{Moree14}. For connections to free numerical semigroups (described in the next section), see \cite[Section 5.1]{CiolanGarciaSanchezMoree2016}.

\section{Explicit descriptions} \label{sec:smooth}

Let $S$ be a free numerical semigroup, as described in \cite[Section 2.2]{Leher2007} and \cite[Chapter 8, Section 4]{RosalesGarciaSanchez09}.  The Ap\'ery sets of free numerical semigroups are well-understood, and so in this context we obtain a further specialization of Theorem \ref{thm:tuenter-apery} in Corollary \ref{cor:smooth-tuenter}.  We then apply our result to the special case of semigroups generated by compound sequences, as were studied in \cite{GassertShor16}.

\subsection{Application to free numerical semigroups}
\begin{defn}
For $k\in\mathbb{N}_0$, suppose $G=(g_0,\dots,g_k)$ is a finite sequence with $g_i\in\mathbb{N}_0$ for all $i$. For $0\le i\le k$, let $G_i=(g_0,\dots,g_i)$ and  $d_i=\gcd(G_i)$. Let $c_i=d_{i-1}/d_i$ for $1\le i\le k$.
If $c_ig_i\in \gens{G_{i-1}}$ for $1\le i\le k$, we say $G$ is a \emph{smooth sequence}. If a set can be ordered to form a smooth sequence, then we say the set is a \emph{smooth set}. We will refer to the sequence $(c_1,\dots,c_k)$ as the \emph{$c$ values of $G$}. 
\end{defn}

\begin{defn}[{\cite[page 133]{RosalesGarciaSanchez09}}]
Suppose $S$ is a numerical semigroup. If $S=\gens{G}$ for some set $G$ with $\gcd(G)=1$ and which can be ordered to form a smooth sequence, then $S$ is a \emph{free numerical semigroup}.
\end{defn}

The class of smooth sequences contains the class of compound sequence, which in turn contains sequences of length two, geometric sequences, and supersymmetric sequences. (References are given in \cite[Section 3]{GassertShor16}.)  The condition that $c_ig_i\in\gens{G_{i-1}}$ for all $i$ provides us with a method (Corollary~\ref{cor:leher-smooth-semigp-elt-repns}) to represent any integer uniquely in terms of elements of $G$, which in turn allows us to identify the elements of $\Ap(S;g_0)$.

We quote the following from \cite{Leher2007} (using slightly different notation here), in which $S$ is generated by a sequence $G$ of positive integers. 
We note that these results do not require $G$ to be minimal (a condition where no proper subset of $G$ generates the same semigroup as $G$), and in fact hold even if $G$ is a sequence of non-negative integers.

\begin{prop}[{\cite[Remark, page 11]{Leher2007}}]\label{prop:leher-smooth-semigp-elt-repns}
Let $G=(g_0,\dots,g_k)$ be a smooth sequence with $c$ values $(c_1,\dots,c_k)$, let $S=\langle G\rangle$, and let $s\in S$. Then $s$ has a unique representation \[s=\sum\limits_{i=0}^k n_ig_i\] with $0\le n_i<c_i$ for $1\leq i\leq k$.
\end{prop}

We extend this result for all $n\in\mathbb{Z}$.

\begin{cor}\label{cor:leher-smooth-semigp-elt-repns} Suppose $G=(g_0,\dots,g_k)$ is smooth with $c$ values $(c_1,\dots,c_k)$. Let $n\in\mathbb{Z}$. Then $n$ has a unique representation \[n=\sum_{i=0}^kn_ig_i\] where $0\leq n_i<c_i$ for $1\leq i\leq k$.
\end{cor}
\begin{proof}
Let $n\in\mathbb{Z}$. If $n\in S$, by Proposition \ref{prop:leher-smooth-semigp-elt-repns}, we are done.

If $n\not\in S$, then there exists some $M\in\mathbb{Z}$ such that $n+Mg_0>F(S)$ and thus $n+Mg_0\in S$. By Proposition \ref{prop:leher-smooth-semigp-elt-repns}, there exist integers $n_1,\dots,n_k$ where \[n+Mg_0=\sum\limits_{i=0}^k n_ig_i,\] with $0\le n_i<c_i$ for $1\le i\le k$, so $n$ has a representation \begin{equation}\label{eqn:representation-of-n}n=(n_0-M)g_0+\sum\limits_{i=1}^k n_i g_i \end{equation} with $0\le n_i<c_i$. Thus, $n$ has at least one representation.

To show this representation is unique, suppose we have another representation $n= \sum\limits_{i=0}^k m_ig_i$ with $0\le m_i<c_i$ for $1\le i\le k$. Then \[n+Mg_0=(M+m_0)g_0+\sum\limits_{i=1}^k m_ig_i.\] Since $n+Mg_0\in S$, it has the unique representation $n+Mg_0=\sum\limits_{i=0}^k n_ig_i$. Therefore for $n+M g_0 = Mg_0+\sum\limits_{i=0}^k m_ig_i$, we have $m_0=n_0-M$ and $m_i=n_i$ for $1\le i\le k$, so this is the same representation as in Equation \eqref{eqn:representation-of-n}.
\end{proof}

\begin{lemma} \label{lemma:leher-smooth-nonnegative-repns}
Suppose $S=\langle G\rangle$ with $G$ smooth. Let $n\in\mathbb{Z}$. Given the unique representation of $n$, we have $n\in S$ if and only if $n_0\ge0$. 
\end{lemma}

\begin{proof}
Clearly if $n_0\ge0$, then $n\in S$. Otherwise if $n_0<0$, then $n \notin S$ by the uniqueness of representation.
\end{proof}

The following proposition is a reformulation of \cite[Theorem 7]{Leher2007}.
\begin{prop}
\label{prop:leher-smooth-apery-repns}
Suppose $S=\langle G\rangle$ for $G$ smooth. Let $n\in\mathbb{Z}$. Given the unique representation of $n$, we have $n\in \Ap(S;g_0)$ if and only if $n_0=0$.
\end{prop}
\begin{proof}
We use Lemma \ref{lemma:leher-smooth-nonnegative-repns} for the cases $n_0<0$, $n_0=0$, and $n_0>0$.

If $n_0<0$ then $n\not\in S$, so $n\not\in\Ap(S;g_0)$. If $n_0=0$, then $n\in S$ and $n-g_0\not\in S$, so $n\in\Ap(S;g_0)$. If $n_0>0$, then $n\in S$ and also $n-g_0\in S$, so $n\not\in\Ap(S;g_0)$.

Thus $n\in\Ap(S;g_0)$ if and only if $n_0=0$.
\end{proof}
That is, if $G$ is smooth then we can explicitly describe the Ap\'ery set of $g_0$ in $S$: \[\Ap(S;g_0)=\left\{\sum_{i=1}^kn_ig_i : 0\le n_i<c_i \right\}.\]  This leads to an explicit version of Theorem~\ref{thm:tuenter-apery} for smooth sequences.

\begin{cor}\label{cor:smooth-tuenter}
Suppose $S$ is a free numerical semigroup, so $S=\gens{G}$ for $G$ a smooth sequence. Then for any $f$ defined on $\mathbb{N}_0$, 
\begin{equation}\label{eqn:smooth-tuenter}
\sum_{n\in\NR}[f(n+g_0)-f(n)] = \sum\limits_{n_1=0}^{c_1-1}\dots\sum\limits_{n_k=0}^{c_k-1} f\left(\sum_{i=1}^k n_ig_i \right) - \sum_{n=0}^{g_0-1}f(n).
\end{equation}
\end{cor}

\subsection{Application to numerical semigroups generated by compound sequences}

For compound sequences, we follow the notation of \cite{GassertShor16}. See also \cite{KiersONeillPonomarenko16}. Compound sequences are special cases of smooth sequences (by Corollary~\ref{cor:compound-seq-are-smooth}), so we apply the above results to compound sequences and obtain an alternate proof of \cite[Theorem 3.3]{GassertShor16}.

\begin{defn}
Let $k\in\mathbb{N}_0$ and $A=(a_1,\dots,a_k),B=(b_1,\dots,b_k)\in\mathbb{N}^k$. We say $(A,B)$ is a \textit{suitable pair} of $k$-tuples if $\gcd(a_i,b_j)=1$ for all $i\ge j$. If $(A,B)$ is a suitable pair, then we let $g_i=b_1\cdots b_i a_{i+1}\cdots a_k$ for $0\le i\le k$.  The sequence $G(A,B)=(g_0,\dots,g_k)$ is a \emph{compound sequence}.  We say a set $G = \{g_0, \ldots, g_k\} \subseteq \mathbb{N}^{k+1}$ is a \emph{compound set} if its elements can be ordered to form a compound sequence.
\end{defn}

\begin{prop}[{\cite[Proposition 2, Property 5]{KiersONeillPonomarenko16}}]
If $(A,B)$ is a suitable pair, then \linebreak $\gcd(G(A,B))=1$, so $\gens{G(A,B)}$ is a numerical semigroup.
\end{prop}

\begin{prop}\label{prop:permute-compound}
Let $G=(g_0,\dots,g_k)=G(A,B)$ be a compound sequence for a suitable pair $(A,B)$ of $k$-tuples with $A=(a_1,\dots,a_k)$, $B=(b_1,\dots,b_k)$. For any $j$ with $0\le j\le k$, let $\rho_j$ be the permutation defined by
\[\rho_j(g_i)=
\begin{dcases*}
g_{j-i} & for $i\leq j$, \\ 
g_i & for $i>j$,
\end{dcases*}\]
so that $\rho_j(G)=(g_j,\dots,g_0,g_{j+1},\dots,g_k)$.
Then $\rho_j(G)$ 
is a smooth (but not necessarily compound) sequence with leading term $g_j$.  The corresponding $c$ values of $\rho_j(G)$ are $
(b_j,\dots,b_1,a_{j+1},\dots,a_k)$.
\end{prop}

\begin{proof}
For notation, let $\widehat{G}=\rho_j(G).$ Then $\widehat{G}=(\widehat{g}_0,\dots,\widehat{g}_k)=(g_j,\dots,g_0,g_{j+1},\dots,g_k).$ In what follows, we will compute $\widehat{d}_i=\gcd(\widehat{g}_0,\dots,\widehat{g}_i)$ and $\widehat{c}_i=\widehat{d}_{i-1}/\widehat{d}_i$.

We find that
\[\widehat{d}_i=\gcd(\widehat{g}_0,\dots,\widehat{g}_i)=
\begin{dcases*}
b_1\cdots b_{j-i}a_{j+1}\cdots a_k & if $0\le i\le j$, \\
a_{i+1}\cdots a_k & if $j<i\le k$,
\end{dcases*}\]
and
\[\widehat{c}_i=\widehat{d}_{i-1}/\widehat{d}_i=
\begin{dcases*}
b_{j-i+1} & if $i\le j$, \\ 
a_i & if $i>j$. 
\end{dcases*} 
\]
In other words, the $c$ values for $\widehat{G}$ are $(\widehat{c}_1,\dots,\widehat{c}_k)=(b_j,\dots,b_1,a_{j+1},\dots,a_k)$.

We now verify the smoothness condition.  Let $\widehat{S}_i=\gens{\widehat{g}_0,\dots,\widehat{g}_i}$ and note that $\widehat{S}_i=S_i$ for $i\ge j$.  To show $\widehat{G}$ is smooth, we need to show $\widehat{c}_i\widehat{g}_i\in \widehat{S}_{i-1}$ for $1\leq i\leq k$.  If $i\leq j$, then $\widehat{c}_i\widehat{g}_i=b_{j-i+1}g_{j-i}=a_{j-i+1}g_{j-i+1}\in\gens{g_j,\dots,g_{j-i+1}}=\gens{\widehat{g}_0,\dots,\widehat{g}_{i-1}}=\widehat{S}_{i-1}$.   If $i>j$, then $\widehat{c}_i\widehat{g}_i=a_ig_i=b_i g_{i-1}\in \gens{g_0,\dots, g_{i-1}}=S_{i-1}=\widehat{S}_{i-1}$.  Thus, $\rho_j(G)=\widehat{G}$ is smooth.
\end{proof}

When $j=0$, $\rho_j(G)=G$, so we see that compound sequences are smooth.
\begin{cor}\label{cor:compound-seq-are-smooth}
The compound sequence $G=G(A,B)$ is smooth with $c$ values $(a_1,\dots,a_k)$.
\end{cor}
\begin{remark}
If $G$ is smooth, $\rho_j(G)$ is not necessarily smooth.
For instance, let $G=(6,10,11)$ and $j=2$.  Then $G$ is smooth (with $c_1=3, c_2=2$), but $\rho_2(G)=(11,10,6)$ (with $\widehat{c}_1=11, \widehat{c}_2=1$) is not smooth.

Similarly, if $G$ is compound, then $\rho_j(G)$ is not necessarily compound.  
As an example, let $A=(2,2)$ and $B=(3,3)$.  Then $G(A,B)=(4,6,9)$ is compound.  Let $j=1$.  Then $\rho_1(G(A,B))=(6,4,9)$, which is not compound.  However, note that $\rho_1(G(A,B))$ is smooth (which it must be by Proposition~\ref{prop:permute-compound}).
\end{remark}

Since $\rho_j(G)$ is smooth by Proposition~\ref{prop:permute-compound}, we can apply Corollary ~\ref{cor:leher-smooth-semigp-elt-repns}, Lemma \ref{lemma:leher-smooth-nonnegative-repns}, and Proposition \ref{prop:leher-smooth-apery-repns} to $\rho_j(G)$ to obtain the following.

\begin{cor}
Let $G=(g_0,\dots,g_k)$ be a compound sequence. Then $G=G(A,B)$ for a suitable pair of $k$-tuples $(A,B)$.  Let $S=\langle G\rangle$.

For any $j$ with $0\le j\le k$ and any $n\in\mathbb{Z}$, there is a unique representation $n=\sum_{i=0}^k n_ig_i$ where $n_0,\dots,n_k\in\mathbb{Z}$ with $0\le n_i<b_{i+1}$ if $0\le i<j$ and $0\le n_i<a_i$ if $j<i\le k$.  Given such a representation, $n\in S$ if and only if $n_j\ge0$. Furthermore, $n\in \Ap(S;g_j)$ if and only if $n_j=0$.
\end{cor}

Combining Corollary~\ref{cor:smooth-tuenter} and Proposition~\ref{prop:permute-compound}, we get the explicit identity for compound sequences which was shown in \cite{GassertShor16}.

\begin{cor}[{\cite[Theorem 3.3]{GassertShor16}}]\label{cor:compound-tuenter}
For a suitable pair $(A,B)$ of $k$-tuples, let $G=G(A,B)=(g_0,\dots,g_k)$ be the corresponding compound sequence for $A=(a_1,\dots,a_k)$ and $B=(b_1,\dots,b_k)$.  Then for any $j$ with $0\le i\le k$ and any $f$ defined on $\mathbb{N}_0$,
\[\sum_{n\in \NR}(f(n+g_j)-f(n)) = \sum_{n_0=0}^{b_1-1}\dots \sum_{n_{j-1}=0}^{b_j-1} \sum_{n_{j+1}=0}^{a_{j+1}-1} \dots \sum_{n_k=0}^{a_k-1} f\left(\sum_{\substack{i=0\\ i\neq j}}^k n_ig_i \right) - \sum_{n=0}^{g_j-1}f(n).\]
\end{cor}

\section{Assorted Sylvester sums from free numerical semigroups} \label{sec:sylvester sum}

We apply Corollary~\ref{cor:smooth-tuenter} to compute certain properties of $\NR$, the complement  of a free numerical semigroup. 

\subsection{Sylvester sums}
\begin{defn}
A numerical semigroup $S$ with Frobenius number $F(S)$ is \emph{symmetric} if for all $n\in\mathbb{Z}$ the set $\{n, F(S)-n\}$ contains exactly one element of $S$.
\end{defn}

\begin{prop}[{\cite[Lemma 2.14 and Corollary 4.5]{RosalesGarciaSanchez09}}]
If $S$ is a numerical semigroup, then $F(S)\le 2g(S)-1$, with equality if and only if $S$ is symmetric.
\end{prop}

If $S$ is a free numerical semigroup, then it is symmetric, and there are explicit formulas for $g(S)$ and $F(S)$.

\begin{prop}[{\cite[Theorem D]{Brauer1962} and \cite[pages 15--16]{Selmer1977}}]\label{prop:brauer-selmer}
Let $S$ be a free numerical semigroup with smooth generating sequence $G=(g_0,\dots,g_k)$ and corresponding $c$ values $(c_1,\dots,c_k)$. Then $S$ is a symmetric numerical semigroup with \[F(S)=-g_0+\sum_{i=1}^k (c_i-1)g_i\] and \[g(S)=\frac{1}{2}\left(1-g_0+\sum_{i=1}^k (c_i-1)g_i\right).\] 
\end{prop}

\begin{defn}
For a numerical semigroup $S$ with generating set $G$, let $\NR(G)=\mathbb{N}_0\setminus S$. The $m$th power Sylvester sum $S_m(G)$ is defined as \[S_m(G)=\sum_{n\in\NR(G)}n^m.\]
\end{defn}

With this notation, the genus is the $0$th power sum: $g(S)=S_0(G)$. If $S$ is symmetric (as free numerical semigroups are), then $F(S)=2S_0(G)-1$.

We can use Corollary~\ref{cor:smooth-tuenter} with $f(n)=n^{m+1}$ to get an explicit formula for $S_m(G)$. We give the results for the first few values of $m$. For notation, given a sequence $G=(g_0,\dots,g_k)$ and some $e\in\mathbb{N}_0$, let $G^e:=(g_0^e,\dots,g_k^e)$.

\begin{prop}\label{prop:sylvester-formulas}
Let $G=(g_0,\dots,g_k)$ be a smooth sequence with corresponding $c$ values $(c_1,\dots,c_k)$. Then
\begin{align*}
S_0(G) &= \frac{1}{2}\left(1-g_0+\sum_{i=1}^k (c_i-1)g_i \right), \\
S_1(G) &= \frac{S_0(G)^2-S_0(G)}{2}+\frac{S_0(G^2)}{12}, \\
S_2(G) &= \frac{2S_0(G)-1}{6}\left(S_0G)^2-S_0(G)+\frac{S_0(G^2)}{2}\right).
\end{align*}
\end{prop}
\begin{proof}
We first note that $S_0(G)=g(S)$, and the formula for $g(S)$ is given in Proposition \ref{prop:brauer-selmer}.

For $S_1(G)$, we follow the proof of \cite[Proposition 3.5]{GassertShor16} which was done for the case of $G$ a compound sequence. Let $f(n)=n^2$.  Then by Corollary~\ref{cor:smooth-tuenter}, 
\[\sum\limits_{n\in \NR}((n+g_0)^2-n^2) = \sum\limits_{n_1=0}^{c_1-1}\dots \sum\limits_{n_k=0}^{c_k-1}
\left(\sum\limits_{i=1}^{k} n_i  g_i\right)^2 - \sum_{n=0}^{g_0 - 1}n^2.\]
Simplifying both sides, we have
\begin{align}\label{eq:main-equation} g_0^2\,S_0(G)+2g_0S_1(G)
=  \sum\limits_{n_1=0}^{c_1-1}\dots \sum\limits_{n_k=0}^{c_k-1}
\left(\sum\limits_{i=1}^{k} n_i  g_i\right)\left(\sum\limits_{j=1}^{k} n_j g_j\right) - \frac{(g_0-1)g_0(2g_0-1)}{6}.\end{align}

Since there are a finite number of summations, each with a finite number of terms, we can change the order to evaluate the summations over $n_1,\dots,n_k$ first.  We will break them up into the cases where $i=j$ and $i\neq j$.  If $i=j$, we get \[\sum_{n_1=0}^{c_1-1}\dots \sum_{n_k=0}^{c_k-1} \sum_{i=1}^{k}n_i^2g_i^2 = g_0 \sum_{i=1}^{k}\frac{(c_{i}-1)(2c_{i}-1)g_i^2}{6}.\]  
If $i\neq j$, we get 
\begin{align*}
\sum_{n_1=0}^{c_1-1}\dots \sum_{n_k=0}^{c_k-1} \sum_{i=1}^{k}n_i g_i\sum_{j=1,j\neq i}^{k} n_j g_j
&= g_0 \sum_{i=1}^{k}\frac{(c_{i}-1)g_i}{2}\sum_{j=1,j\neq i}^{k}\frac{(c_{j}-1) g_j}{2}\\
&=\frac{g_0}{4}\left(\sum_{i=1}^{k}(c_{i}-1)g_i \right)^2 - \frac{g_0}{4}\sum_{i=1}^{k}((c_{i}-1)g_i)^2.
\end{align*}

Note that $\sum_{i=1}^{k}(c_{i}-1)g_i = 
2 S_0(G)+g_0-1.$ Similarly, $\sum_{i=1}^{k}(c_{i}^2-1)g_i^2 = 2 S_0(G^2)+g_0^2-1.$

The right side of equation~\eqref{eq:main-equation} simplifies to 
\[\frac{g_0}{12}(2 S_0(G^2)+g_0^2-1)+\frac{g_0}{4}(2S_0(G)+g_0-1)^2-\frac{(g_{0}-1)g_0(2g_0-1)}{6}.\]  We then subtract $g_0^2S_0(G)$ and divide through by $2g_0$ to obtain the following.
\[S_1(G)=\sum_{n\in \NR} n =\frac{S_0(G)^2-S_0(G)}{2}+\frac{S_0(G^2)}{12}.\]

The approach for $S_2(G)$, which uses $f(n)=n^3$, is identical to the $S_1(G)$ case. Similar to the proof of \cite[Proposition 3.5]{GassertShor16}, we choose to omit the details here because the ideas are the same and the simplifications are more complicated. The reader is encouraged to verify the result.
\end{proof}

As in \cite[Proposition 3.5]{GassertShor16}, we note that $S_2(G)$ is a function of $S_0(G)$ and $S_1(G)$.  In particular, \[S_2(G) = \frac{2S_0(G)-1}{3}\left(3S_1(G)-S_0(G)^2+S_0(G)\right). \]

\subsection{Alternating Sylvester sums}

\begin{defn}
For a numerical semigroup $S$ with generating set $G$, let $\NR(G)=\mathbb{N}_0\setminus S$.  The alternating $m$th power Sylvester sum $T_m(G)$ is defined as 
\[T_m(G)=\sum_{n\in\NR(G)}(-1)^n n^m.\]
\end{defn}

In \cite{WangWang2008}, for the case of $G=\{a,b\}$,  the authors use Corollary~\ref{cor:tuenter} with $f(n)=(-1)^n n^m$ to find recurrences for $T_m(G)$ along with explicit formulas for $m=0,1,2$.  Since we must have $\gcd(G)=1$, we may assume that $b$ is odd and that $a$ is either even or odd.

For notation, given non-negative integers $m$ and $g$, let 
\[\sigma_m(g) = \sum_{n=0}^g n^m \atext{and} \tau_m(g) = \sum_{n=0}^g(-1)^n n^m,\] 
the $m$th power and alternating $m$th power sums of integers in the interval $[0,g]$.

\begin{prop}[{\cite[Theorem 4.1]{WangWang2008}}]\label{prop:wangwang-recursive} 
Let $G=\{a,b\}$.

If $a$ is even and $b$ is odd, then 
\[T_m(G) = \frac{1}{2}\left(\tau_m(b-1)-a^m\sigma_m(b-1)-\sum_{i=0}^{m-1}\binom{m}{i}b^{m-i}T_i(G) \right)\] 
with the initial value $T_0(G)=-\frac{b-1}{2}$.

If $a$ and $b$ are both odd, then \[T_m(G) = \frac{1}{2}\left(\tau_m(b-1)-a^m\tau_m(b-1)-\sum_{i=0}^{m-1}\binom{m}{i}b^{m-i}T_i(G) \right)\] with the initial value $T_0(G)=0$. (Note that $T_m(G)$ is symmetric in $a$ and $b$ when they are both odd.)
\end{prop}

\begin{cor}[{\cite[page 1331]{WangWang2008}}]\label{cor:wangwang-explicit}
When $G=\{a,b\}$, if $a$ is even and $b$ is odd, then
\begin{align*}
T_0(G) &= -\frac{b-1}{2},\\
T_1(G) &=\frac{1}{4}(b-1)(b-ab+1),\\
T_2(G) &=\frac{1}{12}ab(b-1)(a+3b-2ab);
\end{align*}
and if $a$ and $b$ are both odd, then
\begin{align*}
T_0(G) &=0,\\
T_1(G) &=-\frac{1}{4}(a-1)(b-1),\\
T_2(G) &=-\frac{1}{4}ab(a-1)(b-1).
\end{align*}
\end{cor}

We can use Corollary~\ref{cor:smooth-tuenter} to generalize Corollary~\ref{cor:wangwang-explicit} for free numerical semigroups. For $G=(g_0,\dots,g_k)$ smooth, we use $f(n)=(-1)^n n^m$ if $g_0$ is odd and $f(n)=(-1)^n n^{m+1}$ if $g_0$ is even.

\begin{prop}
Let $G=(g_0,\dots,g_k)$ be a smooth sequence with corresponding $c$ values $(c_1,\dots,c_k)$.  Additionally, define $c_0:=1$.
Let $J:=\min\{j : 2\nmid g_j\}$ and let $I_G:=\{i : 2\mid g_i\}.$
Then 
\begin{align*}
T_0(G) &= \frac{1}{2}\left(1-\frac{c_Jg_J}{g_0}\prod_{i \in I_G} c_i \right),\\
T_1(G) &= \frac{(2S_0(G)-1)(2T_0(G)-1)-1}{4},\\
T_2(G) &= \frac{2T_0(G)-1}{12}\left(6S_0(G)^2-6S_0(G)+3S_0(G^2)+g_0^2-(c_Jg_J)^2-\sum_{i\in I_G}g_i^2(c_i^2-1) \right).
\end{align*}
\end{prop}
\begin{proof}
For each $m$, we proceed with two cases depending on whether $g_0$ is odd or even. We present the proof for $m=0$ here. The proofs for $m=1$ and $m=2$ are similar. 

Case 1. Suppose $g_0$ is odd.  We will use $f(n)=(-1)^n$ with Equation~\eqref{eqn:smooth-tuenter}.  The left side simplifies to $-2\sum_{n\in\NR}(-1)^n$.  
To simplify the right side, first note that 
\begin{align*}
\sum\limits_{n_1=0}^{c_1-1}\cdots\sum\limits_{n_k=0}^{c_k-1}(-1)^{n_1g_1+\cdots+n_kg_k} &= \sum\limits_{n_1=0}^{c_1-1}\cdots\sum\limits_{n_k=0}^{c_k-1}\left[(-1)^{n_1g_1}\cdots(-1)^{n_kg_k}\right] \\ 
&= \left(\sum\limits_{n_1=0}^{c_1-1}(-1)^{n_1g_1}\right)\cdots\left(\sum\limits_{n_k=0}^{c_k-1}(-1)^{n_kg_k}\right)\\
&=\prod\limits_{i=1}^k\left(\sum\limits_{n_i=0}^{c_i-1}(-1)^{n_ig_i}\right).
\end{align*}
The right side of Equation~\eqref{eqn:smooth-tuenter} is then \begin{align*}
\sum\limits_{n_1=0}^{c_1-1}\cdots\sum\limits_{n_k=0}^{c_k-1} (-1)^{n_1g_1+\cdots+n_kg_k} - \sum_{n=0}^{g_0-1} (-1)^n & = \prod_{i=1}^k\left(\sum_{n_i=0}^{c_i-1}(-1)^{n_ig_i}\right) - \sum_{n=0}^{g_0-1} (-1)^n.\end{align*} Since $g_0=c_1\cdots c_k$ and $g_0$ is odd, all of the $c_i$ terms are odd, so the right side simplifies to
$-1+\prod_{i\in I_G}c_i.$ Equating the two sides of Equation~\eqref{eqn:smooth-tuenter}, we obtain \[\sum_{n\in\NR}(-1)^n = \frac{1}{2}\left(1-\prod_{i\in I_G}c_i\right) = \frac{1}{2}\left(1-\frac{c_Jg_J}{g_0}\prod_{i\in I_G}c_i\right),\] with the latter equality holding because $J=0$, which means $c_J g_J/g_0=1$.

Case 2. Suppose $g_0$ is even. By the definition of $J$, note that $g_0,\dots,g_{J-1}$ are even and $g_J$ is odd. Therefore since $d_{J-1}$ is even and $d_J$ is odd, $c_J=d_{J-1}/d_J$ is even. We will use $f(n)=n(-1)^n$ with Equation~\eqref{eqn:smooth-tuenter}.  The left side simplifies to $g_0\sum_{n\in\NR} (-1)^n.$  
The right side is 
\[\sum\limits_{n_1=0}^{c_1-1}\cdots\sum\limits_{n_k=0}^{c_k-1}(n_1g_1+\dots+n_kg_k)(-1)^{n_1g_1+\cdots+n_kg_k} - \sum_{n=0}^{g_0-1}n(-1)^n.\] Since $\sum_{n=0}^{g_0-1}n(-1)^n=-g_0/2$ when $g_0$ is even, the right side simplifies to
\begin{align*}
&\left(\sum\limits_{n_1=0}^{c_1-1}n_1g_1(-1)^{n_1g_1}\right)\left(\sum\limits_{n_2=0}^{c_2-1}(-1)^{n_2g_2}\right)\cdots \left(\sum\limits_{n_k=0}^{c_k-1}(-1)^{n_kg_k}\right) \\
&+\left(\sum\limits_{n_1=0}^{c_1-1}(-1)^{n_1g_1}\right)\left(\sum\limits_{n_2=0}^{c_2-1}n_2g_2(-1)^{n_2g_2}\right)\cdots \left(\sum\limits_{n_k=0}^{c_k-1}(-1)^{n_kg_k}\right) \\
&+\cdots\\
&+\left(\sum\limits_{n_1=0}^{c_1-1}(-1)^{n_1g_1}\right)\left(\sum\limits_{n_2=0}^{c_2-1}(-1)^{n_2g_2}\right)\cdots \left(\sum\limits_{n_k=0}^{c_k-1}n_kg_k(-1)^{n_kg_k}\right) \\ &+\frac{g_0}{2},
\end{align*}
which equals 
\[\sum_{j=1}^k\left(\sum_{n_j=0}^{c_j-1}n_jg_j(-1)^{n_jg_j}\prod_{i=1,i\ne j}^k\left(\sum_{n_i=0}^{c_i-1}(-1)^{n_ig_i}\right) \right)+\frac{g_0}{2}.\]
Since $c_J$ is even and $g_J$ is odd, $\sum_{n_J=0}^{c_J-1}(-1)^{n_Jg_J}=0$, so the product above is 0 for all $j\ne J$. The right side simplifies to
\[\sum_{n_J=0}^{c_J-1}n_Jg_J(-1)^{n_J}\prod_{i\in I_G}c_i + \frac{g_0}{2} = -\frac{g_Jc_J}{2}\prod_{i\in I_G}c_i + \frac{g_0}{2}.\] Equating the two sides of Equation~\eqref{eqn:smooth-tuenter}, we obtain \[\sum_{n\in\NR}(-1)^n = \frac{1}{2}\left(1-\frac{c_Jg_J}{g_0}\prod_{i\in I_G}c_i \right). \]
\end{proof}

\begin{remark}
Via Proposition \ref{prop:sylvester-formulas}, we can represent the formula for $T_2(G)$ in a few ways:
\begin{align*}
T_2(G) & = (2T_0(G)-1)\left(3S_1(G)-S_0(G)^2+S_0(G)+\frac{1}{12}\left(g_0^2-(c_Jg_J)^2-\sum_{i\in I_G}g_i^2(c_i^2-1) \right)\right)\\
&=(2T_0(G)-1)\left(\frac{3S_2(G)}{2S_0(G)-1}+\frac{1}{12}\left(g_0^2-(c_Jg_J)^2-\sum_{i\in I_G}g_i^2(c_i^2-1) \right)\right).
\end{align*}
\end{remark}

In particular, since $T_0(G)\le 0$, we see that $\NR$ always contains at least as many odd elements as even elements. The numbers of even and odd elements in $\NR$ are equal exactly when all elements of $G$ are odd.

If $g_0$ is odd, then $g_0=g_Jc_J$ and we can simplify the above formulas.

\begin{cor}
Suppose $G=(g_0,\dots,g_k)$ is a smooth sequence.

If $g_0$ is odd, then 
\begin{align*}
T_0(G)&=\frac{1}{2}\left(1-\prod_{i\in I_G}c_i\right),\\ 
T_1(G)&=-\frac{1}{4}\left(1+(2S_0(G)-1)\prod_{i\in I_G}c_i\right), \\
T_2(G)&=-\left(\frac{3S_2(G)}{2S_0(G)-1}-\frac{1}{12}\left(\sum_{i\in I_G}g_i^2(c_i^2-1)\right)\right)\prod_{i\in I_G}c_i.
\end{align*}  
If all terms of $G$ are odd, then
\begin{align*}
T_0(G)&=0,\\ 
T_1(G)&=-S_0(G)/2, \\
T_2(G)&=-S_0(G)(S_0(G)-1)/2-S_0(G^2)/4 \\
&= -3S_2(G) / (2S_0(G)-1)\\
& = -2T_1(G)^2-T_1(G)+\frac{1}{2}T_1(G^2).
\end{align*}  
\end{cor}

Suppose $G$ contains only odd terms.  As $T_2(G)$ is a function of $S_0(G)$, $S_1(G)$, and $S_2(G)$, it seems reasonable to ask whether in general $T_m(G)$ is a function of $S_0(G),\ldots,S_m(G)$ or even of $S_0(G),\ldots, S_{m-1}(G)$, since $S_2(G)$ is a function of $S_0(G)$ and $S_1(G)$. Equivalently, we ask if $T_m(G)$ is a function of $S_0(G),\dots,S_0(G^m)$ or of $T_1(G), \dots, T_1(G^m)$.

\subsection*{Acknowledgements} The authors would like to thank the anonymous referee for the careful reading of our manuscript and many helpful comments.

\bibliography{main-bib-file}

\begin{thebibliography}{10}

\bibitem{Apery1946}
Roger Ap\'ery.
\newblock Sur les branches superlin\'eaires des courbes alg\'ebriques.
\newblock {\em C. R. Acad. Sci. Paris}, 222:1198--1200, 1946.

\bibitem{Brauer1962}
Alfred Brauer and James~E. Shockley.
\newblock On a problem of {F}robenius.
\newblock {\em Journal f\"{u}r die reine und angewandte {M}athematik},
  211:215--220, 1962.

\bibitem{CiolanGarciaSanchezMoree2016}
Emil-Alexandru Ciolan, Pedro~A. Garc\'{i}a-S\'{a}nchez, and Pieter Moree.
\newblock Cyclotomic numerical semigroups.
\newblock {\em SIAM Journal on Discrete Mathematics}, 30(2):650--668, 2016.

\bibitem{GassertShor16}
T.~Alden Gassert and Caleb~M. Shor.
\newblock On {S}ylvester sums of compound sequence semigroup complements.
\newblock {\em Journal of Number Theory}, 180:45 -- 72, 2017.

\bibitem{KiersONeillPonomarenko16}
Claire Kiers, Christopher O'Neill, and Vadim Ponomarenko.
\newblock Numerical semigroups on compound sequences.
\newblock {\em Comm. Algebra}, 44(9):3842--3852, 2016.

\bibitem{Leher2007}
Eli Leher.
\newblock {\em Applications of the minimal transversal method in numerical
  semigroups}.
\newblock PhD thesis, Tel Aviv University, 2007.

\bibitem{Moree14}
Pieter Moree.
\newblock Numerical semigroups, cyclotomic polynomials, and {B}ernoulli
  numbers.
\newblock {\em The American Mathematical Monthly}, 121(10):890--902, 2014.

\bibitem{RosalesGarciaSanchez09}
J.C. Rosales and P.A. Garc\'{i}a-S\'{a}nchez.
\newblock {\em Numerical Semigroups}, volume~20 of {\em Developments in
  Mathematics}.
\newblock Springer-Verlag New York, 2009.

\bibitem{Selmer1977}
Ernst~S. Selmer.
\newblock On the linear diophantine problem of {F}robenius.
\newblock {\em Journal f\"{u}r die reine und angewandte {M}athematik},
  0293\_0294:1--17, 1977.

\bibitem{Sylvester1882}
J.~J. Sylvester.
\newblock On subvariants, i.e. semi-invariants to binary quantics of an
  unlimited order.
\newblock {\em American Journal of Mathematics}, 5(1):79--136, 1882.

\bibitem{Tuenter06}
Hans~J.H. Tuenter.
\newblock The {F}robenius problem, sums of powers of integers, and recurrences
  for the {B}ernoulli numbers.
\newblock {\em Journal of Number Theory}, 117(2):376 -- 386, 2006.

\bibitem{WangWang2008}
Weiping Wang and Tianming Wang.
\newblock Alternate {S}ylvester sums on the {F}robenius set.
\newblock {\em Computers \& Mathematics with Applications}, 56(5):1328 -- 1334,
  2008.

\end{thebibliography}
\bibliographystyle{plain}
\end{document}